\documentclass[a4paper]{article}
\usepackage{amsmath}
\usepackage{amsthm}
\usepackage[all]{xy}
\usepackage[english]{babel}
\usepackage[utf8]{inputenc}
\usepackage{graphicx}
\usepackage{xypic,color}
\usepackage{amsfonts, amssymb}

\newtheorem{thm}{Theorem}
\newtheorem{defn}[thm]{\textbf{Definition}}
\newtheorem{example}{Example}[section]
\newtheorem{examples}{Examples}[section]
\newtheorem{cor}[thm]{Corollary}
\newtheorem{lem}[thm]{Lemma}
\newtheorem{prop}[thm]{Proposition}
\theoremstyle{definition}

\theoremstyle{definition}

\theoremstyle{remark}
\newtheorem{rem}[thm]{Remark}

\renewcommand{\k}{\Bbbk}

\renewcommand{\d}{\Delta}

\newcommand{\de}{\Delta}
\newcommand{\al}{\alpha}
\newcommand{\be}{\beta}
\newcommand{\ga}{\gamma}
\newcommand{\ot}{\otimes}
\newcommand{\ma}{\mathcal}
\newcommand{\rt}{\rightarrow}

\title{Nearly Frobenius theory and semisimplicity of bimodules}
\author{Dalia Artenstein, Ana Gonz\'alez, Gustavo Mata.
}

\date{\today}
\begin{document}
	\maketitle

	\begin{abstract}
		In the first part of this article we prove that one of the conditions required in the original definition of nearly Frobenius algebra, the coassociativity, is redundant. Also, we determine the Frobenius dimension of the product and tensor product of two nearly Frobenius algebras from the Frobenius dimension of each of them.  We apply these results to semisimple algebras. 
		
		In the second part we introduce the notion of normalized nearly Frobenius algebra. We prove a series of equivalences: the concept of normalized nearly Frobenius algebra is equivalent to the concept of separable algebra, equivalent to the fact that the algebra is projective as a bimodule on itself and, finally, equivalent to the category of bimodules is semisimple. Also, we relate these concepts with the property of semisimplicity of the category of modules over the algebra. 
			
	\end{abstract}
	
	
	
	\section{Introduction}
\label{intro}

The concept of nearly Frobenius algebra is motivated by the result proved in \cite{CG04}, which states that: \emph{the homology of the free loop space $H^*(LM)$ has the structure of a Frobenius algebra without counit}. These objects were studied in \cite{GLSU} and their algebraic properties were developed in \cite{AGL15}, in particular   the possible nearly Frobenius structures in gentle algebras were described.

In the framework of differential graded algebras, Abbaspour considers  in \cite{A}  nearly Frobenius algebras that he calls open Frobenius algebras. He proves that if $A$ is a symmetric open Frobenius algebra of degree $m$, then $HH_{\ast}(A,A)[m]$ is an open Frobenius algebra, where the product at the chain level is given by
$$ a_0\bigl[a_1,\dots, a_n\bigr] \circ b_0\bigl[b_1,\dots, b_m\bigr] =  \left\{\begin{array}{lc}
0& \mbox{if}\; n>0\\
a_0''a_0'b_0\bigl[b_1,\dots, b_m\bigr]
\end{array}\right.
$$ 
and the coproduct is given by
$$\de\left(a_0\bigl[a_1,\dots, a_n\bigr]\right)=\sum_{(a_0),0\leq i\leq n}
\left(a_0'[a_1,\dots, a_{i-1}, a_i]\right)\otimes\left(a_0''[a_{i+1},\dots, a_n]\right).$$

In this work we prove that one of the conditions required in the original definition of nearly Frobenius algebra, the coassociativity, is redundant.  Also, we determine the Frobenius dimension of the product and tensor product of two nearly Frobenius algebras from the Frobenius dimension of each of them.  This applies to the definition of Frobenius algebras too. 

In the second part we introduce the notion of normalized nearly Frobenius algebras, we prove that cartesian and tensor product of normalized nearly Frobenius algebras is also a normalized nearly Frobenius algebra. Later, we prove that the concept of normalized  nearly Frobenius algebra is equivalent to the concept of separable algebra and equivalent in turn to algebra having Hochschild cohomological dimension zero. We give some applications of these results, for example that the matrix algebra is a normalized nearly Frobenius algebra, therefore is separable.
If we consider the category of bimodules over a nearly Frobenius algebra, we prove that the normalized condition over the nearly Frobenius algebra is equivalent that the bimodule category is semisimple. \\  
 The work finish relating the concepts described above with the semisimplicity property of the module category on the nearly Frobenius algebra. 
Although the conclusions presented in the applications are already known the techniques to prove them are originals.


\section{Nearly Frobenius algebras}
\label{sec:nearlyFrob}

\hspace{0.5cm} In one of the definitions of Frobenius algebras it is required that the algebra $A$ admits a coalgebra structure $(A, \Delta,\varepsilon)$ where the coproduct $\Delta$ is a morphism of $A$-bimodules. In the next result we prove that the coassociativity condition is redundant.

\begin{prop}
Let $A$ be a Frobenius algebra, then the coassociativity condition is a consequence of the $A$-bimodule morphism condition of $\Delta$, and the unit axiom.
\end{prop}
\begin{proof}
	In the next diagram we illustrate this affirmation.
	$$\xymatrix{
		A\ar@{=}[d]\ar[rrr]^{\d} & & & A\ot A\ar[ddd]^{1\ot \d}\\
		\k\ot A \ar[r]^{u\ot 1}& A\ot A\ar[r]^{\d\ot 1}\ar[lu]_m\ar[ld]_m\ar[d]_{1\ot\d} & A\ot A\ot A\ar[d]^{1\ot 1\ot \d}\ar[ru]^{1\ot m} & \\
		A\ar@{=}[u]\ar[d]_{\d} & A\ot A\ot A\ar[r]_(.4){\d\ot 1\ot 1}\ar[ld]_{m\ot 1} & A\ot A\ot A\ot A \ar[rd]^{1\ot m\ot 1}& \\
		A\ot A\ar[rrr]_{\d\ot 1} & & & A\ot A\ot A}$$
	All the internal diagrams commute as a consequence of the $A$-bimodule condition, the unit axiom and the natural decomposition of the morphism $\d\ot\d$; then the external diagram commutes too.
\end{proof}

The previous result allows us to give the next alternative definition of nearly Frobenius algebras.

\begin{defn}
	An algebra $A$ is a \emph{nearly Frobenius algebra} if it admits a linear map $\Delta:A\rt A\ot A$ such that
	$$\xymatrix{A\otimes A\ar[r]^{m}\ar[d]_{\Delta\otimes 1}& A\ar[d]^{\Delta}\\
		A\otimes A\otimes A\ar[r]_{1\otimes m}&A\otimes A
	},\quad \xymatrix{A\otimes A\ar[r]^{m}\ar[d]_{1\otimes\Delta}& A\ar[d]^{\Delta}\\
		A\otimes A\otimes A\ar[r]_{m\otimes 1}&A\otimes A
	}$$ commute.
\end{defn}

\begin{defn}
	The \emph{Frobenius space} associated to an algebra $A$ is the vector space of all the possible coproducts $\d$ that make it into a nearly Frobenius algebra ($\mathcal{E}$), see \cite{AGL15}. Its dimension over $\k$ is called the \emph{Frobenius dimension} of $A$, that is,
	$$\operatorname{Frobdim}A = \operatorname{dim}_\k\mathcal{E}.$$
\end{defn}

\begin{defn}
	Let $(A,\Delta_A)$ and $(B,\Delta_B)$ be two nearly Frobenius algebras. A homomorphism $f:A\rightarrow B$ is a \emph{nearly Frobenius homomorphism} if it is a morphism of algebras and the following diagram commutes.
	$$\xymatrix{A\ar[r]^{f}\ar[d]_{\Delta_A}& B\ar[d]^{\Delta_B}\\
		A\otimes A\ar[r]_{f\otimes f}&B\otimes B
	}$$
	If $f$ is bijective then $f$ is said to be an \emph{isomorphism} between $A$ and $B$.\\
	
	Notation: \emph{nFrob} is the category of nearly Frobenius algebras.
\end{defn}

\begin{thm}\label{teo1}
	Let $(A,\Delta_A)$ be a nearly Frobenius algebra, $B$ an algebra and $f:A\rightarrow B$ an isomorphism of algebras. Then $B$ admits a nearly Frobenius structure defined as
	$$\Delta_B=(f\otimes f)\circ\Delta_A\circ f^{-1}.$$
In particular $\operatorname{Frobdim}A=\operatorname{Frobdim}B$.

\end{thm}
\begin{proof}
	We need to check that $\Delta_B$ is a $B$-bimodule morphism. That is,
	$$\xymatrix{B\otimes B\ar[r]^{m}\ar[d]_{\Delta_B\otimes 1}& B\ar[d]^{\Delta_B}\\
		B\otimes B\otimes B\ar[r]_{1\otimes m}&B\otimes B
	},\quad \xymatrix{B\otimes B\ar[r]^{m}\ar[d]_{1\otimes\Delta_B}& B\ar[d]^{\Delta_B}\\
		B\otimes B\otimes B\ar[r]_{m\otimes 1}&B\otimes B
	}$$ commute.
	To prove this we only need to see that the dotted face of the next cube commutes.
	
	$$\xymatrix{
		&B\otimes A\ar@{->}'[d][dd]^{1\otimes\Delta_A}\ar[ld]_{1\otimes f}&A\otimes A\ar@{->}'[d][dd]^{1\otimes\Delta_A}\ar[rr]^m\ar[l]_{{f}\otimes 1}&&A\ar[dd]^{\Delta_A}\ar[dl]_f\\
		B\otimes B\ar@{-->}[dd]_{1\otimes\Delta_B}\ar@{-->}[rrr]^(.6){m}&&&B\ar@{-->}[dd]_(.6){\Delta_B}&\\
		&B\otimes A\otimes A\ar[dl]_{1\otimes f\otimes f}&A\otimes A\otimes A\ar@{->}'[r][rr]^{m\otimes 1}\ar[l]_{f\otimes 1\otimes 1}&&A\otimes A\ar[dl]^{f\otimes f}\\
		B\otimes B\otimes B\ar@{-->}[rrr]_{m\otimes 1}&&&B\otimes B &}
	$$
	Since $f$ is an isomorphism of algebras and $\Delta_A$ is a nearly Frobenius coproduct in $A$ all the other faces commute and then the dotted face commutes.
\end{proof}

\begin{rem}\label{r6}
Assume that $A_1$ and $A_2$ are $\Bbbk$-algebras. The product of the algebras $A_1$ and $A_2$ is the algebra $A = A_1 \times A_2$ with the addition and the
multiplication given by the formulas $(a_1, a_2) + (b_1, b_2) = (a_1 + b_1, a_2 + b_2)$ and $(a_1, a_2)(b_1, b_2) = (a_1b_1, a_2b_2)$, where $a_1$, $b_1\in A_1$ and $a_2$, $b_2\in A_2$. The
identity of $A$ is the element $1 = (1_{A_1}, 1_{A_2}) = e_1+e_2 \in A_1\times A_2$, where $e_1 = (1_{A_1}, 0)$ and $e_2 = (0, 1_{A_2})$. If $\bigl(A_1,\de_1\bigr)$ and $\bigl(A_2,\de_2\bigr)$ are nearly Frobenius algebras then $A$ admits a natural structure of Nearly Frobenius algebra. In the next paragraph we describe this structure.

First, we define $\de(e_1)=\sum(a_1,0)\otimes (a_2,0)$, where $\de_1(1_{A_1})=\sum a_1\otimes a_2$ and $\de(e_2)=\sum (0,b_1)\otimes (0,b_2)$, where $\de_2(1_{A_2})=\sum b_1\otimes b_2$. Then
$$\de(1)=\sum(a_1,0)\otimes (a_2,0)+\sum (0,b_1)\otimes (0,b_2)\in  	A\otimes A.$$
To prove that this defines a bimodule morphism it is necessary to guarantee  that $\de(1)$ satisfies that
$$(c\otimes 1)\de(1)=\de(1)(1\otimes c),\;\forall\; c\in A.$$  
Denote $c=(c_1,c_2)\in A$, then
$$\begin{array}{ccl}
(c\otimes 1)\de(1)&=&\displaystyle{(c_1,c_2)\otimes(1,1)\left[\sum(a_1,0)\otimes (a_2,0)+\sum (0,b_1)\otimes (0,b_2)\right]}\\
&=&\displaystyle{\sum\left((c_1,c_2)\otimes(1,1)\right)\left((a_1,0)\otimes (a_2,0)\right)+\sum \left((c_1,c_2)\otimes(1,1)\right)\left((0,b_1)\otimes (0,b_2)\right)}\\
 &=& \displaystyle{\sum (c_1a_1,0)\otimes (a_2,0)+\sum (0,c_2b_1)\otimes(0,b_2). }
\end{array}$$

On the other hand

$$\begin{array}{ccl}
\de(1)(1\otimes c)&=&\displaystyle{\left[\sum(a_1,0)\otimes (a_2,0)+\sum (0,b_1)\otimes (0,b_2)\right]\left((1,1)\otimes(c_1,c_2)\right)}\\
&=&\displaystyle{\sum\left((a_1,0)\otimes (a_2,0)\right)\left((1,1)\otimes(c_1,c_2)\right)+\sum \left((0,b_1)\otimes (0,b_2)\right)\left((1,1)\otimes(c_1,c_2)\right)}\\
&=& \displaystyle{\sum (a_1,0)\otimes (a_2c_1,0)+\sum (0,b_1)\otimes(0,b_2c_2). }
\end{array}$$

Remember that $\de_{A_1}$ and $\de_{A_2}$ are bimodule morphisms, then 
$$(c_1\otimes 1)\de_{A_1}(1_{A_1})=\sum c_1a_1\otimes a_2=\sum a_1\otimes a_2c_1=\de_{A_1}(1_{A_2})(1\otimes c_1)$$
and
$$(c_2\otimes 1)\de_{A_2}(1_{A_2})=\sum c_2b_1\otimes b_2=\sum b_1\otimes b_2c_2=\de_{A_2}(1_{A_2})(1\otimes c_2)$$

This proves that $(c\otimes 1)\de(1)=\de(1)(1\otimes c)$. Then $A$ is a nearly Frobenius algebra.
\end{rem}

\begin{rem}\label{r7}

Similarly, we can consider the tensor product $A\otimes B$ of the  $\Bbbk$-algebras $A$ and $B$. As before, we can define a nearly Frobenius coproduct on $A\otimes B$. In this case we take the transposition map $\tau: (A\otimes A)\otimes(B\otimes B)\rightarrow (A\otimes B)\otimes (A\otimes B)$ and the coproduct on $A$ and $B$ to define the coproduct on $A\otimes B$ as follows
$$\de:=\tau\circ \de_A\otimes\de_B :A\otimes B\rightarrow (A\otimes B)\otimes (A\otimes B).$$
Since all the maps are linear, the map $\de$ is linear too. We will test only one of the two necessary conditions to guarantee that it is bimodule morphism, the other one is analogous. 
$$\xymatrix @C=5pc{
a\otimes b\otimes c\otimes d\ar@{|->}[r]^{\tau^{-1}} \ar@{|->}[d]_{\de_A\otimes\de_B\otimes 1_{A\otimes B}}& a\otimes c\otimes b\otimes d\ar@{|->}[d]^{m_{A\otimes B}}\\
    \de_A(a)\otimes\de_B(b)\otimes c\otimes d\ar@{=}[d] & ac\otimes bd\ar@{|->}[d]^{\de_A\otimes \de_B}\\
    \sum a_1\otimes a_2\otimes\sum b_1\otimes b_2\otimes c\otimes d\ar@{|->}[d]_{\tau\otimes 1_{A\otimes B}}&   \de_A(ac)\otimes\de_B(bd)\ar@{=}[d]\\  
\sum\sum a_1\otimes b_1\otimes a_2\otimes b_2\otimes c\otimes d\ar@{|->}[d]_{1_{A\otimes B}\otimes\tau^{-1}}& \sum a_1\otimes a_2c\otimes\sum b_1\otimes b_2d\ar@{|->}[d]^{\tau}\\    
\sum\sum a_1\otimes b_1\otimes a_2\otimes c\otimes b_2\otimes d\ar@{|->}[r]_{1\otimes m_A\otimes m_B}&  \sum\sum a_1\otimes b_1\otimes a_2c\otimes b_2d
}$$ 

\end{rem}

\begin{prop}\label{P7}
	Consider $A$ and $B$ two $\Bbbk$-algebras, then the following isomorphisms of vector spaces hold:
	\begin{enumerate}
		\item $\mathcal E_{A\times B}\cong \mathcal E_A\times\mathcal E_B$. In particular $\operatorname{Frobdim}(A\times B)=\operatorname{Frobdim}(A)+\operatorname{Frobdim}(B)$.
		\item $\mathcal E_{A\otimes B}\cong \mathcal E_A\otimes\mathcal E_B$. Therefore  $\operatorname{Frobdim}(A\otimes B)=\operatorname{Frobdim}(A).\operatorname{Frobdim}(B)$.
	\end{enumerate}	
\end{prop}
\begin{proof}
	In Remark \ref{r6} we saw that there exist natural inclusions of $\mathcal E_A\times\mathcal E_B$ in $\mathcal E_{A\times B}$ and, in Remark \ref{r7}, of $\mathcal E_A\otimes\mathcal E_B$ in $\mathcal E_{A\otimes B}$.\\
	To finish the proof it is necessary to check that the maps are surjective. 
	\begin{enumerate}
		\item We note the unit of $A\times B$ as  $1=e_1+e_2$, where $e_1=\bigl(1_{A},0\bigr)$ and $e_2=\bigl(0,1_{B}\bigr).$ \\
		Let's take $\Delta\in \mathcal E_{A\times B}$ and express $\Delta(1)$ as follows:
		$$\de(1)=\sum_{i,j}(\eta_i,\xi_i)\otimes(\rho_j,\nu_j)$$
		with $\eta_i, \rho_j\in A$ and $\xi_i, \nu_j\in B$ for all $i,j$.
		Since $\Delta$ is a bimodule morphism we can prove that $\Delta(e_1)=\sum_{i,j}(\eta_i,0)\otimes(\rho_j,0)$ and, in a similar way, that $\Delta(e_2)=\sum_{i,j}(0,\xi_i)\otimes(0,\nu_j)$. Then, we conclude that the coproduct has the expression
		$$\Delta(1)=\Delta(e_1)+\Delta(e_2)=\sum_{i,j}(\eta_i,0)\otimes(\rho_j,0)+\sum_{i,j}(0,\xi_i)\otimes(0,\nu_j).$$
		This allows us to define $\Delta_A(1_A)=\sum_{i,j}\eta_i\otimes\rho_j$ and $\Delta(1_B)=\sum_{i,j}\xi_i\otimes\nu_j$. Using again that $\Delta$ is a bimodule morphism, we deduce that $\Delta_A$ and $\Delta_B$ are also bimodule morphisms, then $\bigl(A,\de_A\bigr)$ and $\bigl(B,\de_B\bigr)$ are nearly Frobenius algebras. In particular $\Delta=\iota\circ\bigl(\Delta_A+\Delta_B\bigr)$.
		\item Consider $\Delta\in \mathcal E_{A\otimes B}$ and $\{x_i\}$,  $\{y_j\}$ bases of $A$ and $B$ respectively, where $x_1=1_A$ and $y_1=1_B$. Then
		$$\Delta(1_A\otimes 1_B)=\sum_{i,j,k,l}a_ib_jc_kd_lx_i\otimes y_j\otimes x_k\otimes y_l,\quad\mbox{where}\; a_i, b_j, c_k,d_l\in\Bbbk.$$
		Using that $\Delta$ is bimodule morphism we have that
		$$\Delta(x\otimes 1_B)=\sum_{i,j,k,l}a_ib_jc_kd_lxx_i\otimes y_j\otimes x_k\otimes y_l=\sum_{i,j,k,l}a_ib_jc_kd_lx_i\otimes y_j\otimes x_kx\otimes y_l.$$
		As $y_j=y_l=1_B$ when $j=l=1$ we can define 
		$$\Delta_A(1_A)=\sum_{i,k}a_ic_kx_i\otimes x_k,$$
		analogously we can define
		$$\Delta_B(1_B)=\sum_{j,l}b_jd_ly_j\otimes y_l.$$
		Note that with these definitions the coproduct $\Delta$ is
		$$\Delta=(1\otimes\tau\otimes 1)\circ(\Delta_A\otimes\Delta_B).$$
	\end{enumerate}
\end{proof}

The next corollary is a consequence of Theorem \ref{teo1} and Proposition \ref{P7}.

\begin{cor}
	If $A$ is a semisimple algebra over an algebraically closed field $\Bbbk$, then it is possible to determine completely its Frobenius dimension.
\end{cor}
\begin{proof}
	By the Artin-Wedderburn Theorem we have that $A\cong M_{n_1}(\Bbbk)\times  M_{n_2}(\Bbbk)\times\cdots\times  M_{n_r}(\Bbbk)$. Then
	$$\operatorname{Frobdim}(A)=\operatorname{Frobdim}(M_{n_1}(\Bbbk)\times  M_{n_2}(\Bbbk)\times\cdots\times  M_{n_r}(\Bbbk))=\sum_{i=1}^r \operatorname{Frobdim}(M_{n_i}(\Bbbk))=\sum_{i=1}^rn_i^2.$$
	Finally 
	$$\operatorname{Frobdim}(A)=\sum_{i=1}^rn_i^2.$$
\end{proof}  

\begin{cor}
	Let $G$ be a finite group. If $\operatorname{char}(\Bbbk)$ does not divide the order of $G$ and $\Bbbk$ is an algebraically closed field, then it is possible to determine completely the Frobenius dimension of $\Bbbk G$.
\end{cor}
\begin{proof}
	Applying Maschke’s theorem we have that $\Bbbk G$ is a semisimple algebra then, by the previous corollary,  it is possible to determine completely its Frobenius dimension.
\end{proof}

In the next results we are going to use an example presented in \cite{AGL15}, which has a small error in its calculation.  We shall now present the result quoted and its correction.\\

Let $G$ be a cyclic finite group of order $n$ and the group algebra $\k G$, with the natural basis $\bigl\{g^i:\;i=1,\dots,n\bigr\}$. This algebra is a nearly Frobenius algebra. Moreover, we can determine all the nearly Frobenius structures that it admits.

Using the bimodule condition of the coproduct, we can prove that a basis of the Frobenius space is
$$\ma{B}=\bigl\{\de_k:\k G\rt \k G\ot \k G: k\in{1,\dots, n}\bigr\},$$
where $\displaystyle{\de_1(1)=\sum_{i=1}^ng^i\ot g^{n+1-i}}$ and $\displaystyle{\de_k\bigl(1\bigr)=\sum_{i=1}^{k-1}g^i\ot g^{k-i}+\sum_{i=k}^ng^i\ot g^{n+k-i}}$ for $k=2,\dots, n$. \\In particular, we have that
$$\operatorname{Frobdim}(\k G)=\left|G \right|.$$

The general expression of any nearly Frobenius coproduct in the unit is
$$\de\bigl(1\bigr)=a_1\sum_{i=1}^ng^i\ot g^{n+1-i}+\sum_{k=2}^n a_k\left(\sum_{i=1}^{k-1}g^i\ot g^{k-i}+\sum_{i=k}^ng^i\ot g^{n+k-i}\right),$$ where $a_i\in\Bbbk$ for $i=1,\dots,n$.

\begin{cor}
	If $G$ is a finite abelian group, then it is possible to determine $\operatorname{Frobdim}(\Bbbk G)$.
\end{cor}
\begin{proof}
	If $G$ is a finite abelian group, then $G=G_1\oplus G_2\oplus\cdots\oplus G_p$, where $G_i$ is a finite cyclic group for $i\in \{1,\dots ,p\}$. The group algebra $\Bbbk G$ is isomorphic, as a $\Bbbk$-algebra, to $\Bbbk G_1\otimes \Bbbk G_2\otimes\cdots\otimes \Bbbk G_p$. Therefore, applying Theorem \ref{teo1} and Proposition \ref{P7}
	$$\operatorname{Frobdim}(\Bbbk G)=\operatorname{Frobdim}(\Bbbk G_1\otimes \Bbbk G_2\otimes\cdots\otimes \Bbbk G_p)=\prod_{i=1}^p\operatorname{Frobdim}(\Bbbk G_i)=\prod_{i=1}^p\left|G_i \right|. $$
	Finally,
	$$\operatorname{Frobdim}(\Bbbk G)=\prod_{i=1}^p\left|G_i \right|.$$
\end{proof}

\begin{examples}\label{e1}
	We illustrate the results given in Proposition \ref{P7} with a couple of examples.	
	\begin{enumerate}
		\item Let's consider the cyclic groups $G$ and $H$ where $\left|G \right|=2$ and $\left|H \right|=3$ and their corresponding group algebras $A_1=\Bbbk G$, $A_2=\Bbbk H$. Then, by Proposition \ref{P7},1., $B=A_1\times A_2$ is a nearly Frobenius algebra of Frobenius dimension $5$.\\
		$$\mathcal{E}_{A_1}=\operatorname{span}_\Bbbk\left\{\de_1^1, \de_2^1\right\}, $$ where $\de_1^1(1)=g\otimes 1+1\otimes g$,and $\de_2^1(1)=g\otimes g+1\otimes 1$.
		$$\mathcal{E}_{A_2}=\operatorname{span}_\Bbbk\left\{\de_1^2, \de_2^2, \de_3^2 \right\},$$ where $\de_1^2(1)=h\otimes 1+h^2\otimes h^2+1\otimes h$, $\de_2^2(1)=h^2\otimes 1+h\otimes h+1\otimes h^2$, and $\de_3^2(1)=h\otimes h^2+1\otimes 1+h^2\otimes h$. Therefor
		$$\mathcal{E}_{B}=\operatorname{span}_\Bbbk\left\{\de_1, \de_2, \de_3, \de_4, \de_5 \right\},$$ where
		$\de_1(1,1)=\bigl(\de_1^1(1),0\bigr)$, 	$\de_2(1,1)=\bigl(\de_2^1(1),0\bigr)$,	$\de_3(1,1)=\bigl(0, \de_1^2(1)\bigr)$, $\de_4(1,1)=\bigl(0, \de_2^2(1)\bigr)$ and 	$\de_5(1,1)=\bigl(0, \de_3^2(1)\bigr)$.\\
		Then, the general expression of any nearly Frobenius coproduct in the unit is
		$$\de\bigl(1,1\bigr)=\left(a_1\de_1^1(1_G)+a_2\de_2^1(1_G), b_1\de_1^2(1_H)+b_2\de_2^2(1_H)+b_3\de_3^2(1_H)\right),$$ where $a_i, b_j\in\Bbbk$ for $i=1,2$ and $j=1,2,3$.
	
				\item 
	Consider the linear quiver  $Q: \xymatrix{\underset{1}{\bullet}\ar[r]^\eta  &  \underset{2}{\bullet}}$ and its associated path algebra: $$A=\Bbbk Q=\langle e_1, e_2, \eta\rangle.$$ It is known that $\mathcal{E}_A$ is a vector space of dimension 1, and a generator is the coproduct $\Delta:A\rightarrow A\ot A$ defined as  $$\Delta(1)=\eta\ot e_1+e_2\ot\eta.$$
	Now we will construct the tensor product of two copies of $A$: 
	$$B=A\otimes A=\langle e_1\ot e_1, e_1\ot e_2, e_2\ot e_1, e_2\ot e_2, e_1\ot\eta, \eta\ot e_1, \eta\ot e_2, e_2\ot\eta, \eta\ot\eta\rangle.$$
	This algebra admits only one coproduct, and it is $$\overline{\Delta}=(1\otimes\tau\otimes 1)\circ(\Delta\otimes\Delta).$$

	On the other hand, if we consider the next quiver
	$$
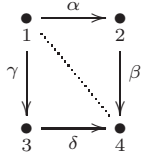
\xymatrix{
		\underset{1}{\bullet}\ar[r]^\al\ar[d]_{\gamma} \ar@{..}[dr] &  \underset{2}{\bullet}\ar[d]^{\be}\\
		\underset{3}{\bullet}\ar[r]_\delta  &  \underset{4}{\bullet}
	}$$
	and the algebra $\displaystyle{C=\frac{\Bbbk Q}{\langle\al\be-\ga\delta\rangle}=\langle \mathbf{e_1}, \mathbf{e_2},\mathbf{ e_3}, \mathbf{e_4}, \al, \be, \gamma, \de, \al\be\rangle}$ we can prove that this algebra is isomorphic to $B$. The isomorphism  given on the basis is as follows:\\
	$$\varphi: B\rightarrow C$$ 
	$$\varphi\bigl(e_1\ot e_1\bigr)=\mathbf{e_1},\quad \varphi\bigl(e_1\ot e_2\bigr)=\mathbf{e_2},\quad \varphi\bigl(e_2\ot e_1\bigr)=\mathbf{ e_3},\quad \varphi\bigl(e_2\ot e_2\bigr)=\mathbf{e_4}$$
	$$\varphi\bigl(e_1\ot\eta\bigr)=\al,\quad\varphi\bigl(\eta\ot e_2\bigr)=\be,\quad\varphi\bigl(\eta\ot e_1\bigr)=\gamma,\quad \varphi\bigl(e_2\ot\eta\bigr)=\delta$$
	$$\varphi\bigl(\eta\ot\eta\bigr)=\al\be.$$
	It is clear that the isomorphism $\varphi$ respects the algebra
	structures. Then, we can conclude that $\mathcal{E}_C$ has dimension one and a generator is:  
	$$
	\de(\textbf{1})=\al\be\ot \mathbf{e_1}+\be\ot\al+\delta\ot\ga+\mathbf{e_4}\ot\al\be.
	$$
	\end{enumerate}	
\end{examples}

\begin{rem}
	In these lines we want to make notice that we cannot establish a nice property that relates the Frobenius dimension of a quotient algebra with the original algebra.
	First, in the Example 7 of \cite{AGL15} a nontrivial coproduct is constructed in the quotient algebra $A/J$ from a nontrivial structure on $A$, but $\operatorname{Frobdim}(A)=1$ and $\operatorname{Frobdim}(A/J)=3$.\\
	In addition, we can not always recover a nontrivial structure on the quotient from one on the original algebra, for example if we consider $A=\Bbbk \mathbb{A}_4 = \langle e_1, e_2, e_3, e_4, \alpha, \beta, \gamma \rangle$ with all the arrows having the same orientation and the radical square zero algebra $B=\Bbbk \mathbb{A}_4/I$, we know that $A$ admits only one nontrivial nearly Frobenius coproduct, that is
	$$\Delta\bigl(e_1\bigr)=\alpha\beta\gamma\otimes e_1,\, \Delta\bigl(e_2\bigr)=\beta\gamma\otimes\alpha,\, \Delta\bigl(e_3\bigr)=\gamma\otimes\alpha\beta,\, \Delta\bigl(e_4\bigr)=e_4\otimes\alpha\beta\gamma,$$
	and this structure is trivial in $B$. But we can prove that $B$ admits nontrivial nearly Frobenius coproducts, moreover $\operatorname{Frobdim}B=5$.  	 
\end{rem}
\vspace{0.3cm}
In the next paragraph we will give a nice interpretation of the Frobenius space of an algebra $A$ using hochschild cohomology.
\begin{rem}
	
	 For an $A$-bimodule $M$, where $A$ is an algebra we call
	$$M^A =\bigl\{m\in M: a m=m a\quad\forall a\in A\bigr\}$$
	the sub–bimodule of invariants. In Remark 1 of \cite{AGL15} is shown that every nearly coproduct is determined by its value in $1$, that is, we have a linear isomorphism
	$$\Phi:\mathcal{E}_A\rightarrow (A\otimes A)^A,\;\mbox{such that}\; \Phi(\de)=\de(1).$$
	Moreover, if we remember that the $0$-group of Hochschild cohomology of $M$ with coefficients in $A$ is
	$$H^0(A,M)=\bigl\{m\in M: am=ma\;\forall a\in A\bigr\}.$$
	In particular for $M=A\otimes A$ we have 
	$$\mathcal{E}_A\cong (A\otimes A)^A=H^0(A, A\otimes A).$$
	Then, it is possible to identify the Frobenius space of $A$ with the 0-group of Hochschild cohomology of $A\otimes A$ with coefficients in $A$.
	
\end{rem}

\vspace{.5cm}


\subsection{Normalized nearly Frobenius algebras}
In the following results we will restrict ourselves to work with a subfamily of nearly Frobenius algebras. This construction is motivated by the notion of normalized Fourier transform (see \cite{walter}).

\begin{defn}
	Let $A$ be an algebra and $\de$ a nearly Frobenius coproduct, we say that $\de$ is {\em normalized} if $m\circ\de=\operatorname{Id}_A$, where $m$ is the product of $A$. If $A$ admits a normalized Frobenius coproduct  we will say that $(A, \de)$ is a normalized nearly Frobenius algebra. 
\end{defn}

\begin{example}
	Let $G$ be a cyclic finite group. The group algebra $A=\k G$ is a nearly Frobenius algebra. We can consider the nearly Frobenius coproduct $\de(1)=\frac{1}{|G|}\sum_{k=1}^ng^k\ot g^{n-k}$. It is a simple verification that $\de$ is normalized.
\end{example}

\begin{prop}\label{prop16}
	\begin{enumerate}
		\item If $A$ and $B$ are nearly Frobenius algebras with normalized coproducts, then $C=A\times B$ has normalized coproduct. 
		\item If $A$ and $B$ are nearly Frobenius algebras with normalized coproducts, then $D=A\otimes B$ has normalized coproduct. 
	\end{enumerate}
\end{prop}
\begin{proof}
	By previous results we know that $C$ and $D$ are nearly Frobenius algebras. We only need to prove that the induced coproducts are normalized.
	\begin{enumerate}
		\item In the first case the coproduct is defined as
		$$\de(c)=\sum (c_1a_1,0)\otimes (a_2,0)+\sum (0,c_2b_1)\otimes(0,b_2),$$
		where  $c=(c_1,c_2)\in C$ and $\de_A(1_A)=\sum a_1\otimes a_2$, $\de_B(1_B)=\sum b_1\otimes b_2$. Then
		$$m\circ \de(c)=\sum (c_1a_1a_2,0)+\sum (0,c_2b_1b_2)$$
		Using that $\de_A$ and $\de_B$ are normalized we have that $\sum c_1a_1a_2=c_1$ and $\sum c_2b_1b_2=c_2$, therefore  $m\circ \de(c)=(c_1,c_2)=c$.
		\item In the second case the coproduct is defined as  
		$$\de_D:=\tau\circ \de_A\otimes\de_B :D=A\otimes B\rightarrow (A\otimes B)\otimes (A\otimes B),$$
		where 	$\tau: (A\otimes A)\otimes(B\otimes B)\rightarrow (A\otimes B)\otimes (A\otimes B)=D\otimes D$ is the transposition map. With this notation the product in $D$ can be described as
		$$m_D=(m_A\otimes m_B)\circ\tau^{-1}:(A\otimes B)\otimes (A\otimes B)\rightarrow A\otimes B.$$
		Then\\
		$$\begin{array}{rcl}
		 m_D\circ \de_D &= &(m_A\otimes m_B)\circ\tau^{-1}\circ \tau\circ \bigl(\de_A\otimes\de_B\bigr)\\
		                                    &= &(m_A\otimes m_B)\circ\bigl(\de_A\otimes\de_B\bigr)\\
		                                    &= &\bigl(m_A\circ\de_A\bigr)\otimes \bigl(m_B\circ\de_B\bigr)\\
		                                    &= &id_A\otimes id_B= id_D.
		 \end{array}$$
		 This concludes the proof that the coproduct of $D$ is normalized.
	\end{enumerate}
\end{proof}


\subsection{Separable algebras}
In this section, we present known results about separable algebras and we study their relationship with the notion of normalized nearly Frobenius algebras.

A good reference for this section is \cite{assem}.

\begin{defn}
Let $R$ be a commutative ring. An $R$-algebra $A$ is called {\em separable} 	if the multiplication map
	$$m: A\otimes_R  A \rightarrow A$$	
	has a section $\sigma$ (i.e. $m\circ\sigma = Id_A$) which is an $A$-bimodule homomorphism.
\end{defn}

\begin{prop}
	Let $R$ be a commutative ring and let $A$ be a separable $R$-algebra. Given a section $\sigma$ of $m$ which is an $A$-bimodule homomorphism, set\\
	$e = \sigma(1)$ and write $\displaystyle{e=\sum^n_{i=1}x_i\otimes_R y_i}$ for suitable $n\in \mathbb{N}$ and $x_i$, $y_i\in A$ for every $i=1,\dots, n$.
	Then we have
	\begin{enumerate}
		\item[(1)] $m(e)=1$\quad i.e.\quad $\displaystyle{\sum^n_{i=1}x_iy_i=1}$.
		\item[(2)]  $ae = ea$\quad i.e.\quad $\displaystyle{\sum^n_{i=1}ax_i\otimes_R y_i=\sum_{i=1}^nx_i\otimes y_ia}$ for every $a\in A$.
	\end{enumerate}
	 \end{prop}
 
 \begin{defn}
 	Let $A$ be an algebra over a commutative ring $R$. An element $e\in A\otimes_R A$ is called a {\em separability element} for $A$ (over $R$) if $e$ fulfills (1) and (2).
 \end{defn}

\begin{prop}
	Let $A$ be an algebra over a commutative ring $R$. Then\\
	$A$ is a separable $R$-algebra $\Leftrightarrow$ $A\otimes_R A$ contains a separability element for $A$ over $R$.\\
	Moreover any separability element of $A$ is an idempotent element of the ring $A\otimes_R A^{op}$.
\end{prop}

\begin{prop}
	Let $A$ be an algebra over a field $\Bbbk$. If $A$ is separable over $\Bbbk$, then $\operatorname{dim}_\Bbbk(A) < \infty$.
\end{prop}

\begin{thm}
Let $A$ be an algebra over a field $\Bbbk$ of finite dimension. Then, the following conditions are equivalent.
		\begin{enumerate}
			\item[(1)] $A$ has Hochschild cohomological dimension $0$.
			\item[(2)] $A$ is projective as $A^e$-module.
			\item[(3)] $A$ has a separability element.
		\end{enumerate}
\end{thm}

\begin{thm}
	An algebra $A$ admits a normalized nearly Frobenius coproduct if and only if $A$ is separable.
\end{thm}	
\begin{proof}
	If $\de$ is a normalized nearly Frobenius coproduct then $\de(1)\in A\ot A$ such that  $a\de(1)=\de(1)a$ $\forall a\in A$, and $m\bigl(\de(1)\bigr)=1$. Therefore $e=\de(1)$ is a separability element thus $A$ is separable. 		\\
	
	Conversely, let be $e$ a separability element, first note that $ae=ea$ for all $a\in A$ then it induces a nearly Frobenius coproducto $\de$ such that $\de(1)=e$. 
	
	Finally, the condition $m(e)=1$ say that $m\bigl(\de(1)\bigr)=1$  then $m\bigl(\de(a)\bigr)=m\bigl(a\de(1)\bigr)=am\bigl(\de(1)\bigr)=a1=a$, $\forall a \in A$, thus $m\circ \de=\operatorname{Id}_A$. Therefore $\de$ is a normalized nearly Frobenius coproduct.
\end{proof}

\begin{rem}
	Note that we prove, in particular, that every separable algebra is a nearly Frobenius algebra. Moreover, the concept of nearly Frobenius coproduct or nearly Frobenius algebra is a weakening of the concept of separable algebra.
	
\end{rem}

\begin{example}
Every Azumaya algebra is nearly Frobenius algebra. Remember that an $R$-algebra $A$ is said to be an {\em Azumaya $R$-algebra} if $A$ is both central and 	separable over $R$. See \cite{Ford}.
\end{example}

An immediate consequence of Proposition 18, Theorem 21 and Theorem 22 is the following result.

\begin{prop}
$A$ admits a normalized nearly Frobenius coproduct if and only if $A$ has Hochschild cohomological dimension $0$.
\end{prop}

\begin{defn}
		An algebra $A$ over an algebraically closed field $\Bbbk$ is called \emph{semisimple} if $A$ is finite dimensional and every left $A$-module is projective.
\end{defn}

The next result can be proved using Theorem 4.5.7 of \cite{Ford} together with the theorem of Artin-Wedderburn.
\begin{prop}
	Let $A$ be a separable algebra over a field $\Bbbk$, then $A$ is semisimple. If $\Bbbk$ is a perfect field then the concepts are equivalents. 
\end{prop}

\begin{cor}
	Every normalized nearly Frobenius algebra is semisimple.
\end{cor}

\begin{rem}
	Note that we prove, in Corollary 9, that if $\Bbbk$ is algebraically closed then every semisimple algebra over $\Bbbk$ is nearly Frobenius.
Now, using Proposition 25, we have that every semisimple algebra is separable if  $\Bbbk$  is perfect. Then, the result of Corollary 9 can be refined in the following way. If $\Bbbk$ is a perfect field and $A$ is a $\Bbbk$-algebra, then $A$ is a semisimple algebra if and only if $A$ is a normalized nearly Frobenius algebra. 	 
\end{rem}

\subsubsection{Applications}

The following results are known, but this paper presents another way of proving them using the previously determined Frobenius structures.

\paragraph{Matrix algebra:}

If we consider the matrix algebra $M_n(\k)$, one nearly Frobenius coproduct of this algebra is $$\de(I)=\frac{1}{n}\sum_{i,k=1}^nE_{ik}\ot E_{ki}$$ for this coproduct we have $\bigl(m\circ\de\bigr)(I)=I$. Then  $M_n(\k)$ is separable, in particular, if $\Bbbk$ is an algebraically closed field, then $M_n(\k)$ is semisimple.

\paragraph{Group algebra:}

In $A=\k G$, where $G$ is a cyclic finite group, we can define the nearly Frobenius coproduct $$\de(1)=\frac{1}{|G|}\sum_{k=1}^ng^k\ot g^{n-k}$$
Note that
$$\bigl(m\circ\de\bigr)(1)=\frac{1}{|G|}\sum_{k=1}^ng^k g^{n-k}=\frac{1}{|G|}\sum_{k=1}^ng^n=1$$
Then $A=\k G$ is separable, and semisimple if $\Bbbk$ is an algebraically closed field.

\begin{example}
	Retaking the example 1 of Examples \ref{e1}, using the Proposition \ref{prop16}, we can see that $B=A_1\times A_2$  admits a normalized coproduct then it is separable. Remember that $A_1=\Bbbk G$ and $A_2=\Bbbk H$, where $G$ and $H$ are cyclic groups of order $2$ and $3$ respectively. 
\end{example}

\paragraph{Truncated polynomial algebra:}

Let $A=\frac{\k[x]}{\langle x^{n+1} \rangle}$ be a nearly Frobenius algebra where a basis of the Frobenius space is $$\ma{E}=\bigl\{\de_0,\de_1,\dots,\de_{n}\bigr\}$$
where
$$\de_k\bigl(1\bigr)=\sum_{i+j=n+k}x^i\ot x^j,\quad\mbox{for}\; k\in\{0,\dots,n\}.$$
Then, $\displaystyle{\de=\sum_{k=0}^n a_k\de_k}$, where $a_k\in\k$ for all $k\in\{0,\dots,n\}$, is a general nearly Frobenius coproduct.

It is easy to prove that there is not a normalized copoduct: $$m\circ\de(1)=\sum_{k=0}^n a_k\sum_{i+j=n+k}x^ix^j=\sum_{k=0}^n a_k\sum_{i+j=n+k}x^{n+k}=a_0\sum_{i+j=n}x^{n}=(n+1)a_0x^n\neq 1.$$ 

Then $A=\frac{\k[x]}{\langle x^{n+1} \rangle }$ is not separable, for all $n\geq 1$.

\paragraph{Path algebra:}

Finally we consider the path algebra generated by the quiver
$$ Q: \xymatrix{\underset{1}{\bullet}\ar[r]^\alpha  &  \underset{2}{\bullet}\ar[r]^\beta & \underset{3}{\bullet}}$$
$A=\k Q=\bigl\langle e_1,e_2,e_3,\al,\be,\al\be\bigr\rangle$. $\operatorname{Frobdim}(A)=1$ and the nearly Frobenius coproduct is
$$\de(e_1)=\al\be\ot e_1,\quad \de(e_2)=\be\ot\al,\quad\de(e_3)=e_3\ot\al\be,$$
$$\de(\al)=\al\be\ot\al,\quad\de(\be)=\be\ot\al\be,\quad\de(\al\be)=\al\be\ot\al\be$$
Note that $m\circ\de(e_1)=\al\be e_1=0$, then is not normalized. Therefore $A=\k Q$ is not separable.

\subsection{Bimodule category on normalized nearly Frobenius algebras}

In this section we study the relationship between the normalized nearly Frobenius structure on an algebra and its category of bimodules.

First we give a technical result that will be used later.

\begin{lem}\label{lemma1}
	An object $M\in {}_A\ma{M}_{A}$ is projective if and only if the structure morphism $\chi:A\ot M\ot A\rt M$ splits in $_A\ma{M}_A$.
\end{lem}
\begin{proof}
	The direct of the assertion is a consequence of the fact that the algebra has unit.
	
	Now we suppose that $\chi:A\ot M\ot A\rt M$ splits in $_A\ma{M}_A$, then exists $\nu:M\rt A\ot M\ot A$ such that $\chi\circ\nu=\operatorname{Id}_M$ in $_A\ma{M}_A$.\\
	As $\nu$ is a homomorphism in $_A\ma{M}_A$ the next diagram commutes
	$$\xymatrix@R=2pc@C=2.5pc{A\ot M\ot A\ar[r]^{\chi}\ar[d]_{1\ot\nu\ot 1}&M\ar[d]^{\nu}\\A\ot A\ot M\ot A\ot A\ar[r]_(.6){m\ot 1\ot m}&A\ot M\ot A}$$
	Let be $P,\, Q\in {}_A\ma{M}_{A}$, $f:M\rt Q$ homomorphism and $g:P\rt Q$ epimorphism in ${}_A\ma{M}_{A}$.
	$$\xymatrix{&M\ar[d]^f\\P\ar[r]^g&Q\ar[r]&0}$$
	If we only consider  the linear structure in the previous diagram we can affirm that there exist a linear map $\al:Q\rt P$ such that $g\circ\al=\operatorname{Id}_Q$. Using this map we define the map $h:M\rt P$ as the composition
	$$\xymatrix{M\ar[r]^(.3)\nu&A\ot M\ot A\ar[r]^{1\ot f\ot 1 }&A\ot Q\ot A\ar[r]^{1\ot\al\ot 1}&A\ot P\ot A\ar[r]^(.7){\chi_P}&P}$$
	First we prove that $g\circ h=f$:
	$$\xymatrix{
		M\ar[r]^(.3)\nu\ar@{=}[rd]& A\ot M\ot A\ar[r]^{1\ot f\ot 1}\ar[d]_{\chi}&A\ot Q\ot A\ar[r]^{1\ot\al\ot 1}\ar@{=}[rd]&A\ot P\ot  A\ar[r]^(.7){\chi_P}\ar[d]^{1\ot g\ot 1}&P\ar[d]^g\\
		&M\ar@/_2pc/[rrr]^f&&A\ot Q\ot A\ar[r]^(.7){\chi_Q}&Q
	}$$
	The last step is to prove that $h$ is a homomorphism in ${}_A\ma{M}_{A}$, i.e. the diagram
	$$\xymatrix{
		A\ot M\ot A\ar[r]^(.6){\chi}\ar[d]_{1\ot h\ot 1}&M\ar[d]^h\\
		A\ot P\ot  A\ar[r]_(.6){\chi_P}& P
	}$$
	commutes.\\
	Note that $h$ is a composition of homomorphisms in  ${}_A\ma{M}_{A}$.
	$$h=\chi_P\circ(1\otimes\al\otimes 1)\circ(1\otimes f\otimes 1)\circ\nu.$$

\end{proof}

Let be $A$ a $\k$-algebra and $m:A\ot A\rt A$ the product of this algebra. Then $\bigl(A, {}_Am_{A}\bigr)\in {}_A\ma{M}_{A}$
where ${}_Am_{A}:A\ot A\ot A\rt A$ is $m(m\ot 1)=m(1\ot m)$.\\

The following theorem is the central result of this section that allows to relate the normalized nearly Frobenius algebras with their bimodules.

\begin{thm}\label{theorem2}
	The object $\bigl(A, {}_Am_{A}\bigr)\in {}_A\ma{M}_{A}$ is projective if and only if $A$ admits a normalized nearly Frobenius coproduct.
\end{thm}
\begin{proof}
	If $\bigl(A, {}_Am_{A}\bigr)$ is projective bimodule, then, by Lemma \ref{lemma1}, there exists $\al: A\rt A\ot A\ot A$ a homomorphism in ${}_A\ma{M}_{A}$ such that ${}_Am_{A}\circ\al=\operatorname{Id}_A$.
	
	We define $\de:A\rt A\ot A$ as $\de=(1\ot m)\circ\al$.\\
	First note that the normalized condition is immediate: $$\displaystyle{m\circ\de=m\circ(1\ot m)\circ\al={}_Am_{A}\circ\al=\operatorname{Id}_A}.$$
	To prove that $\de$ is an $A$-bimodule homomorphism we need to check that $\bigl(m\ot 1 \bigr)\bigl(1\ot\de\bigr)=\de\circ m=\bigl(1\ot m\bigr)\bigl(\de\ot 1\bigr)$.
	
	Remember that $\al$ is a homomorphism in ${}_A\ma{M}_{A}$, then the next diagram commutes
	$$\xymatrix{A\ot A\ar[r]^m\ar[d]_{1\ot \al}&A\ar[d]^{\al}\\A\ot A\ot A\ot A\ot A\ar[r]_(.6){m\ot 1\ot 1}&A\ot A\ot A}$$
	This implies that
	$$\xymatrix{A\ot A\ar[r]^(.4){1\ot\al}\ar[d]_m& A\ot A\ot A\ot A\ar[r]^(.55){1\ot 1\ot m}\ar[d]^(.6){m\ot 1\ot 1}&A\ot A\ot A\ar [d]^{m\ot 1}\\
		A\ar[r]_(.3){\al}&A\ot A\ot A\ar[r]_{1\ot m}&A\ot A}$$  commutes.  Then $(m\ot 1)\circ(1\ot\de)=\de\circ m$.
	
	Applying the associativity of the product and the fact that $\al$ is a homomorphism in ${}_A\ma{M}_{A}$, in particular in $\ma{M}_A$, we have that the next diagram commutes
	$$\xymatrix{A\ot A\ar[r]^(.4){\al\ot 1}\ar[d]_m& A\ot A\ot A\ot A\ar[r]^(.55){1\ot m\ot 1}\ar[d]^(.6){1\ot 1\ot m}&A\ot A\ot A\ar [d]^{1\ot m}\\
		A\ar[r]_(.3){\al}&A\ot A\ot A\ar[r]_{1\ot m}&A\ot A}$$ Then $(1\ot m)\circ(\de\ot 1)=\de\circ m$.
	
	Reciprocally, let be $\de:A\rt A\ot A$ a normalized nearly Frobenius coproduct. Applying the Lemma \ref{lemma1} we need to prove that $_Am_A$ split, i.e. there exists a morphism $\, \al:A\rt A\ot A\ot A$ in ${}_A\ma{M}_A$ such that ${}_Am_A\circ\al=\operatorname{Id}_A$.\\
	We define $\al=(\de\ot 1)\circ\de=(1\ot\de)\circ\de$. Note that
	$${}_Am_A\circ\al=m\circ(1\ot m)(1\ot\de)\circ\de=m\circ(1\ot 1)\circ\de=m\circ\de=\operatorname{Id}_A.$$
	To finish we need to prove that $\al$ is a homomorphism in ${}_A\ma{M}_{A}$, i.e. the next diagram commutes
	$$\xymatrix{
		A\ot A\ot A\ar[r]^{{}_Am_A}\ar[d]_{1\ot\al \ot 1}&A\ar[d]^\al\\
		A\ot A\ot A\ot A\ot A\ar[r]_(.6){m\ot 1\ot m}&A\ot A\ot A
	}$$
	$$\xymatrix@C=3pc{
		A\ot A\ot A\ar[r]^{1\ot\de\ot 1}\ar[d]_{m\ot 1}&A\ot A\ot A\ot A\ar[d]_{m\ot 1\ot 1}\ar[r]^{1\ot\de\ot 1\ot 1}&A\ot A\ot A\ot A\ot A\ar[d]^{m\ot 1\ot 1}\\
		A\ot A\ar[r]^{\de\ot 1}\ar[d]_m&A\ot  A\ot A\ar[r]^{\de\ot 1\ot 1}\ar[d]_{1\ot m}&A\ot A\ot A\ot A\ar[d]^{1\ot 1\ot m}\\
		A\ar[r]_{\de}&A\ot A\ar[r]_{\de\ot 1} &A\ot A\ot A
	}$$
	The internal diagrams commute by the nearly Frobenius property of the coproduct $\de$. Then the external diagram commutes.
\end{proof}

\begin{cor}
	Let be $A$ a $\k$-algebra. Then, the next conditions are equivalent
	\begin{enumerate}
		\item[$(1)$] $A$ admits a normalized nearly Frobenius algebra.
		\item[$(2)$] $\bigl(A,{}_Am_A\bigr)$ is projective in ${}_A\ma{M}_A$.
		\item[$(3)$] Every $\bigl(M,\rho_M\bigr)\in {}_A\ma{M}_A$ is projective.
		\item[$(4)$] The category ${}_A\ma{M}_A$ is semisimple (every $\bigl(M,\rho_M\bigr)\in {}_A\ma{M}_A$ is semisimple).
	\end{enumerate}
\end{cor}
\begin{proof}
	$(1)\Leftrightarrow (2)$ is Theorem \ref{theorem2}.\\
	$(3)\Rightarrow (2)$ It is immediate.\\
	$(2)\Rightarrow (3)$ To prove that $\bigl(M,\rho_M\bigr)$ is projective is equivalent, by the Lemma \ref{lemma1}, to prove that $\rho_M$ split, i.e. there exists $\al_M:M\rt A\ot M\ot A$ in ${}_A\ma{M}_A$ such that $\rho_M\circ\al_M=\operatorname{Id}_M$.\\
	As $\bigl(A,{}_Am_A\bigr)$ is projective, by the Theorem \ref{theorem2}, there exist $\de:A\rt A\ot A$ nearly Frobenius coproduct normalized. Then we define the map $\al_M$ as the composition
	$$\xymatrix@C=3pc{M\ar[r]^(.3){u\ot 1\ot u}&A\ot M\ot A\ar[r]^(.4){\de\ot 1\ot\de}&A\ot A\ot M\ot A\ot A\ar[r]^(.6){1\ot\rho_M\ot 1}&A\ot M \ot A}$$
	First we prove that $\rho_M\circ\al_M=\operatorname{Id}_M$:
	$$\xymatrix@C=3pc{M\ar[r]^(.3){u\ot 1\ot u}\ar@/_3pc/[drrr]_{\operatorname{Id}_M}&A\ot M\ot A\ar[r]^(.4){\de\ot 1\ot\de}\ar@{=}[dr]&A\ot A\ot M\ot A\ot A\ar[r]^(.6){1\ot\rho_M\ot 1}\ar[d]^{m\ot 1\ot m}&A\ot M \ot A\ar[d]^{\rho_M}\\
		&&A\ot M\ot A\ar[r]_{\rho_M}&M}$$
	Finally we need to prove that $\al_M$ is un homomorphism in ${}_A\ma{M}_A$ i.e. the next diagram commutes
	$$\xymatrix@C=3pc{
		A\ot M\ot A\ar[r]^{\rho_M}\ar[d]_{1\ot\al_M\ot 1} &M\ar[d]^{\al_M}\\ A\ot A\ot M\ot A\ot A\ar[r]_(.6){m\ot 1\ot m}& A\ot M\ot A
	}$$
	\scalebox{0.9}{$$\xymatrix@C=3pc{
		a\ot m\ot b\ar@{|->}[r]^{\rho_M}\ar@{|->}[dd]_{1\ot u\ot 1\ot u\ot 1}& amb\ar@{|->}[r]^(.4){u\ot 1\ot u}&1\ot amb\ot 1\ar@{|->}[d]^{\de\ot 1\ot\de}\\
		&&\sum\xi_1\ot \xi_2\ot amb\ot\xi_1\ot\xi_2\ar@{|->}[d]^{1\ot\rho_M\ot 1}\\
		a\ot 1\ot m\ot 1\ot b\ar@{|->}[d]_{1\ot\de\ot 1\ot\de\ot 1}&& \sum\xi_1\ot\xi_2 amb\xi_1\ot\xi_2\\
		\sum a\ot\xi_1 \ot \xi_2 \ot m\ot\xi_1\ot\xi_2\ot b\ar@{|->}[r]_(.55){1^2\ot\rho_M\ot 1^2}&\sum a \ot\xi_1\ot \xi_2m\xi_1\ot\xi_2\ot b\ar@{|->}[r]_(.55){m\ot 1\ot m}&\sum a\xi_1\ot\xi_2m\xi_1\ot\xi_2b\ar@{=}[u]
	}
	$$}

	The expressions $\sum a\xi_1\ot\xi_2m\xi_1\ot\xi_2b$ and $\sum\xi_1\ot\xi_2 amb\xi_1\ot\xi_2$ agree by the Frobenius condition of the coproduct $\de$.\\
	$(3)\Leftrightarrow (4)$ It is a classic result in representation theory (see, for example, \cite{assem}).
\end{proof}

\begin{rem}
If we look the examples of section 2.2, we can conclude that the categories of bimodules over $A=M_n(\k)$ and $B=\k G$ are semisimple, but the category of bimodules over $C=\frac{\k[x]}{\langle x^{n+1} \rangle}$ is not semisimple.
\end{rem}

\vspace{0.3cm}
	Finally, with everything developed we relate the studied with the category of modules on a nearly Frobenius algebra.
\begin{thm}
	If $\Bbbk$ is a perfect field we have the following sequence of equivalences:
	$${}_A\ma{M}\;\, \mbox{is semisimple}\Leftrightarrow\; A\;\,\mbox{is semisimple} \Leftrightarrow\; A\;\,\mbox{is separable} \Leftrightarrow {}_A\ma{M}_A\;\,\mbox{is semisimple}.$$
	
If $\Bbbk$ is not a perfect field we have the following sequence of implications:	
$${}_A\ma{M}_A\;\,\mbox{is semisimple} \Leftrightarrow\; A\;\,\mbox{is separable} \Rightarrow\; A\;\,\mbox{is semisimple}\Leftrightarrow\; {}_A\ma{M}\;\, \mbox{is semisimple}.$$
\end{thm}

\end{document}